\documentclass{amsart}

\usepackage[T1]{fontenc}
\usepackage{amssymb}
\usepackage{url}
\usepackage{mathrsfs}
\usepackage[dvipsnames]{xcolor}

  \usepackage[pdftex]{hyperref}
\hypersetup{
    colorlinks=true, 
    linktoc=all,     
    linkcolor=blue,  
    allcolors=blue,
}

\makeatletter
\@namedef{subjclassname@2020}{\textup{2020} Mathematics Subject Classification}
\makeatother

\newtheorem{theorem}[subsection]{Theorem}

\newtheorem{lemma}[subsection]{Lemma}

\newtheorem{claim}[subsection]{Claim}

\newtheorem{question}[subsection]{Question}

\theoremstyle{definition}
\newtheorem{definition}[subsection]{Definition}

\theoremstyle{remark}
\newtheorem*{remark}{Remark}
\newtheorem{example}{Example}

\newcommand{\omegaone}{{\omega_1}}

\newcommand{\ch}{\mathsf{CH}}
\newcommand{\zfc}{\mathsf{ZFC}}
\newcommand{\ma}{\mathsf{MA}}
\newcommand{\diamondplus}{{\lozenge^+}}

\newcommand{\cohengen}[2]{\mathsf{Add}({#1}, {#2})}

\newcommand{\set}[1]{\left\{ {#1} \right\}}

\newcommand{\type}{\mathrm{type}}
\newcommand{\dom}{\mathrm{dom}}
\newcommand{\ot}{\mathrm{ot}}
\newcommand{\depth}{\mathrm{depth}}
\newcommand{\width}{\mathrm{width}}
\newcommand{\pfa}{\mathsf{PFA}}
\renewcommand{\P}{\mathbb{P}}

\newcommand\restr[2]{{
  \left.\kern-\nulldelimiterspace 
  #1 
  \vphantom{\big|} 
  \right|_{#2} 
  }}

\title{Hajnal--Máté graphs, Cohen reals, and disjoint type guessing}
\author{Chris Lambie-Hanson}
\address{Institute of Mathematics of the Czech Academy of Sciences, \v{Z}itná 25, 115 67 Prague 1, Czech Republic}
\email{lambiehanson@math.cas.cz}
\urladdr{https://math.cas.cz/lambiehanson/}
\author{Dávid Uhrik}
\address{Charles University, Faculty of Mathematics and Physics, Department of Algebra, Sokolovská 83, 186 75 Praha 8, Czech Republic}
\address{Institute of Mathematics of the Czech Academy of Sciences, \v{Z}itná 25, 115 67 Prague 1, Czech Republic}
\email{uhrik@math.cas.cz}

\thanks{The first author was supported by GA\v{C}R project 23-04683S, and 
the second author was supported by Charles University Research Center program No.\ UNCE/SCI/022. Both authors were supported by the Academy of Sciences of the Czech Republic (RVO 67985840).}

\begin{document}
\maketitle

\begin{abstract}
A Hajnal--M\'{a}t\'{e} graph is an uncountably chromatic graph on $\omega_1$ satisfying 
a certain natural sparseness condition.
We investigate Hajnal--M\'{a}t\'{e} graphs and generalizations thereof, focusing
on the existence of Hajnal--M\'{a}t\'{e} graphs in models resulting from 
adding a single Cohen real. In particular, answering a question of D\'{a}niel Soukup, we show that such 
models necessarily contain triangle-free Hajnal--M\'{a}t\'{e} graphs. In the process, 
we isolate a weakening of club guessing called \emph{disjoint type guessing} that we 
feel is of interest in its own right. We show that disjoint type guessing is independent of 
$\zfc$ and, if disjoint type guessing holds in the ground model, then the forcing extension 
by a single Cohen real contains Hajnal--M\'{a}t\'{e} graphs $G$ such that the chromatic numbers 
of finite subgraphs of $G$ grow arbitrarily slowly.
\end{abstract}

\section{Introduction}
András Hajnal and Attila Máté, in \cite{hajnalmate1975}, initiated the 
study of a class of graphs on $\omegaone$ which satisfy two properties 
that are in tension with one another. First, they are sparse in the following sense: the set of neighbors of any countable ordinal $\alpha$ restricted to ordinals below $\alpha$ is finite or cofinal in $\alpha$ with order type $\omega$. Second, 
despite this sparseness condition, they have uncountable chromatic number. 
Such graphs are called \emph{Hajnal--M\'{a}t\'{e} graphs} (or \emph{HM graphs}). See Definition \ref{hm_def} below
for a more precise, and more general, definition.

The existence of HM graphs turns out to be independent of $\zfc$. 
The first existence result is due to Hajnal and Máté \cite{hajnalmate1975}. They showed that under $\diamondplus$ an HM graph exists. In the same paper, they showed that under Martin's Axiom ($\ma(\omegaone)$) there are no such graphs.

Since then, many new constructions have been discovered. Let us mention a few. Komjáth, in his series of papers about HM graphs \cite{komjath1980, komjath1984, komjath1989}, showed that one can construct a triangle-free HM graph just from the $\diamondsuit$ principle. From $\diamondplus$, he constructed an HM graph with no \emph{special cycles}, i.e., cycles formed from two monotone paths. Komjáth and Shelah \cite{komjathshelah1988} constructed further examples of HM graphs, both 
through forcing constructions and through the use of $\diamondsuit$. Given a natural number $s$, they constructed an HM graph having no odd cycles of length less than or equal to $2s+1$ for which the complete bipartite graph $K_{\omega, \omega}$ is not a subgraph. Using $\diamondsuit$, Lambie-Hanson and Soukup \cite{lambiehansonsoukup2021} constructed a coloring $c: [\omegaone]^2 \to \omega$ such that $c^{-1}\{n\}$ is a triangle-free HM graph for each $n$.

In this paper, we continue the investigation of HM graphs, focusing in 
particular on the existence of interesting HM graphs 
after adding a single Cohen real to an arbitrary model of set theory.
There is a rich body of work investigating the effect of adding a single Cohen real 
on the existence of combinatorial structures of size $\omega_1$.
The earliest results in this direction are due to Roitman, who showed 
in \cite{roitman1979} that $\ma(\omega_1)$ necessarily fails after adding a single Cohen real. This result was strengthened
 by Shelah, who proved in \cite{shelah1984} that adding a single Cohen real adds a Suslin tree. It was known that a similar construction can be used to construct an HM graph in a model obtained by adding one Cohen real. Dániel Soukup asked \cite[Problem 5.2]{soukuphandout} whether this HM graph can be triangle-free. Here, we provide a positive answer. In particular, in 
Theorem \ref{cohenandhmgraph} we will show that, in the forcing extension 
by a single Cohen real, for every natural number $s$, there is an HM graph 
with no odd cycles of length $2s+1$ or shorter and no special cycles 
(and hence, e.g., no copies of $K_{\omega,\omega}$).

We also investigate a generalization of HM graphs in which the vertex 
set is not necessarily $\omega_1$ but is rather some tree $T$ of height 
$\omega_1$. Such graphs were first considered by Hajnal and Komj\'{a}th 
in \cite{hajnalkomjath1984}. Here, we construct simple $\zfc$ examples of generalized HM graphs possessing the properties discussed in the 
previous paragraph. In particular, for every natural number $s$, we 
construct in $\zfc$ an HM graph on the tree ${^{<\omega_1}}\omega$ that 
has no odd cycles of length $2s+1$ or shorter and no special cycles.

In the process of proving these results, we isolate a weakening of the 
classical \emph{club guessing} principle that we call \emph{disjoint type 
guessing} ($\mathsf{DTG}$), and which seems to be of interest in its own right. Certain weak 
forms of $\mathsf{DTG}$ are provable in $\zfc$ and are crucial ingredients in 
the proofs of the results mentioned in the previous two paragraphs. The full 
$\mathsf{DTG}$ is independent of $\zfc$; for instance, we show in 
Theorem \ref{pfa_thm} that a weak form of the Proper Forcing Axiom 
entails its failure. However, we show that if the ground model satisfies 
$\mathsf{DTG}$, then one can improve Theorem \ref{cohenandhmgraph}. In 
particular, in the extension by a single Cohen real, there exist 
HM graphs for which the chromatic numbers of their finite subgraphs 
grow arbitrarily slowly with respect to the size of their vertex sets 
(see Section \ref{growth_rate_sec} for a more precise formulation). 
Such HM graphs were first constructed through a forcing construction 
by Komj\'{a}th and Shelah in \cite{komjathshelah2005} and from 
$\diamondsuit$ by the first author in \cite{lambiehanson2020b}, 
solving a question of Erd\H{o}s, Hajnal, and Szemer\'{e}di 
\cite{erdos1982}.

The structure of the paper is as follows. In Section \ref{type_sec}, 
we introduce the notion of \emph{type guessing} that will play 
a central role in the proofs of our main theorems, and prove some 
basic facts about it. In Section \ref{main_result_sec}, we prove 
our main result (Theorem \ref{cohenandhmgraph}) about the existence of HM graphs with no short odd cycles and no special cycles in the extension 
by a single Cohen real. We also prove our $\zfc$ result (Theorem 
\ref{thmgraphinzfc}) about the existence of generalized HM graphs 
indexed by trees and having no short odd cycles or special cycles.
In Section \ref{growth_rate_sec}, we prove that if disjoint type guessing 
holds in the ground model, then in the extension by a single Cohen 
real there is an HM graph such that the chromatic numbers of its 
finite subgraphs grow arbitrarily slowly. In Section \ref{pfa_sect}, 
we show that $\mathsf{DTG}$ is independent of $\zfc$ by showing that 
its negation follows from a weakening of the Proper Forcing Axiom. 
Finally, in Section \ref{q_sec}, we conclude with some remaining open 
questions.

\subsection*{Notation}
We use standard set theoretic notation. If $X$ is a set and $\mu$ a cardinal, then $[X]^\mu := \set{Y \subseteq X \mid |Y| = \mu}$. 
If $X$ and $Y$ are subsets of some ordered set, then $X<Y$ means that each element of $X$ lies below each element of $Y$.

We identify each cardinal with the least ordinal of its cardinality, and 
each ordinal is identified with the set of ordinals strictly less 
than it. In particular, $\omega_1$ is the set of all countable ordinals. The class of ordinals is denoted $\mathrm{Ord}$. If $a$ 
is a well-ordered set (e.g., $a \subseteq \mathrm{Ord}$), 
then $\ot(a)$ denotes the order type of $a$.
We will often identify sets of ordinals with the functions giving their 
increasing enumerations. For example, if $a \subseteq \mathrm{Ord}$ and 
$i < \ot(a)$, then $a(i)$ is the unique $\beta \in a$ such that $\ot(a 
\cap \beta) = i$, and, if $I \subseteq \ot(a)$, then $a[I] := \{a(i) \mid 
i \in I\}$. If $a \subseteq \mathrm{Ord}$ and $\beta \in \mathrm{Ord}$, then 
we will write, e.g., $a < \beta$ instead of $a < \{\beta\}$ to denote that 
every element of $a$ is less than $\beta$. If $a$ and $b$ are sets of ordinals, then 
we let $a \sqsubseteq b$ denote the assertion that $a$ is an initial segment of $b$, 
i.e., $a \subseteq b$ and, for all $\beta \in b$, either $\beta \in a$ or 
$a < \beta$. If $\beta \in \mathrm{Ord}$, 
then $\lim(\beta)$ denotes the set of limit ordinals that are less than $\beta$ 
(we adopt the convention that $0$ is not a limit ordinal).

A graph $G$ is a pair $(X,E)$, where $X$ (the set of \emph{vertices}) is an arbitrary set and $E$ (the set of \emph{edges}) is a subset of $[X]^2$.
If $v \in X$, then $\mathrm{N}_G(v)$ denotes the set of neighbors of $v$ 
in $G$, i.e., $\mathrm{N}_G(v) := \{u \in V \mid \{u,v\} \in E\}$. If $V$ is ordered by $\prec$, then $\mathrm{N}_G^\prec(v) := \{ u \prec v \mid \{u,v\} \in E\}$. If the graph $G$ is clear from context, it will be omitted from 
this notation.

If $\mu$ is a cardinal, then a \emph{proper coloring of a graph $G = (X,E)$ with $\mu$ colors} is a function $c:X \to \mu$ such that 
$c(u) \neq c(v)$ whenever $\set{u,v} \in E$. The chromatic number of a graph $G$, denoted $\chi(G)$, is the least cardinal $\mu$ such that there is a proper coloring of $G$ with $\mu$ colors.

Given a graph $G = (X,E)$, a natural number $n\ge3$, and an injective sequence $(v_0, \ldots, v_{n-1})$ of vertices in $G$, we say that $(v_0, \ldots, v_{n-1}, v_0)$ forms a \emph{cycle of length $n$} in $G$ if $\{v_i, v_{i+1}\} \in E$ for each $i<n-2$ and $\{v_{n-1}, v_0\} \in E$. 
An \emph{odd cycle} is simply a cycle whose length is an odd integer. 
Recall that a graph is \emph{bipartite} if and only if it has no odd 
cycles. If $\mu$ and $\nu$ are cardinals, then $K_{\mu, \nu}$ denotes the 
complete bipartite graph with parts of cardinality $\mu$ and $\nu$.

We use $\cohengen{\omega}{1}$ to denote the forcing to add a single Cohen real. The underlying set of $\cohengen{\omega}{1}$ is ${^{<\omega}}\omega$, 
i.e., the collection of all functions from a natural number to $\omega$. 
$\cohengen{\omega}{1}$ is ordered by reverse inclusion. 
If $g$ is a generic filter over $\cohengen{\omega}{1}$, then 
$r = \bigcup g$ is a function from $\omega$ to $\omega$; it is 
referred to as \emph{the Cohen real} in the forcing extension 
$V[g]$. We refer the reader to \cite{kunen1980} for an introduction to forcing and independence proofs in set theory.

\section{Type guessing} \label{type_sec}

In this section, we introduce the basic notions of \emph{disjoint 
types} and \emph{disjoint type guessing}. Informally, disjoint types are finite binary strings coding 
the order relations between two disjoint sets of ordinals of the same size. 
More formally:

\begin{definition}
	Suppose that $n$ is a natural number.
	\begin{enumerate}
		\item A \emph{disjoint type of width $n$} is a function $t : 2n \rightarrow 2$
		such that
		\[
		|\{i < 2n \mid t(i) = 0\}| = |\{i < 2n \mid t(i) = 1\}| = n.
		\]
        The width of a disjoint type $t$ is denoted $\width(t)$.
		\item If $a$ and $b$ are disjoint elements of $[\mathrm{Ord}]^n$, then
		$\type(a,b)$ is the unique disjoint type $t :2n \rightarrow 2$ such that,
		letting $a \cup b = \{\alpha_0, \alpha_1, \ldots, \alpha_{2n-1}\}$,
		enumerated in increasing order, we have $a = \{\alpha_i \mid i < 2n \text{ and } t(n) = 0\}$
		and $b = \{\alpha_i \mid i < 2n \text{ and } t(n) = 1\}$.
        \item Suppose that $t$ is a disjoint type of width $n$ and $a,b \in 
        [\mathrm{Ord}]^n$ are such that $\type(a,b) = t$. The \emph{depth} of $t$, 
        denoted $\depth(t)$ is the least $k < n$ such that either $a < b(k)$ or 
        $b < a(k)$.
        \item If $t$ is a disjoint type of width $n$, we let $\bar{t}$ denote 
        the \emph{opposite} type of $t$, i.e., $\bar{t}:2n \rightarrow 2$ is 
        defined by letting $\bar{t}(i) = 1-t(i)$ for all $i < 2n$. Note that, 
        if $a,b \in [\mathrm{Ord}]^n$, then $\type(a,b) = t$ if and only if 
        $\type(b,a) = \bar{t}$. Note also that $\depth(\bar{t}) = \depth(t)$.
	\end{enumerate}
\end{definition}

We will often represent disjoint types of width $n$ as binary 
strings of length $2n$ in the obvious way. If $t_0$ and $t_1$ are 
two disjoint types of width $n_0$ and $n_1$, respectively, 
then the concatenation $t_0{}^\frown t_1$ is the disjoint type of width 
$n_0 + n_1$ represented by the concatenation of the binary strings 
representing $t_0$ and $t_1$. Formally, $t_0{}^\frown t_1$ is given by 
\[
  t_0{}^\frown t_1(i) = \begin{cases}
      t_0(i) & \text{if } i < 2n_0; \\ 
      t_1(i - 2n_0) & \text{if } 2n_0 \leq i < 2n_0 + 2n_1.
  \end{cases}
\]

We will be particularly interested in the following
family of types, sometimes known as \emph{Specker types}.

\begin{definition}
	Suppose that $s$ and $n$ are natural numbers with $1 \leq s < n$. Then
	$t^n_s$ is the disjoint type of width $n$ defined by letting, for all
	$i < 2n$,
	\[
		t^n_s(i) =
		\begin{cases}
			0 & \text{if } i < s \\
			0 & \text{if } s \leq i < 2n - s \text{ and } i - s \text{ is even} \\
			1 & \text{if } s \leq i < 2n - s \text{ and } i - s \text{ is odd} \\
			1 & \text{if } i \geq 2n - s.
		\end{cases}
	\]
\end{definition}

\begin{example}
  Let us illustrate the above definitions with an example. In its 
  binary string representation, $t^n_s$ is the disjoint type of width 
  $n$ starting with $s$ copies of $0$, followed by $n-s$ copies of $01$, 
  followed by $s$ copies of 
  $1$. For example, $t^5_2$ is represented in this way as $0001010111$. One can 
  see from this example that $\depth(t^5_2) = 2$. In general, $\depth(t^n_s) = n-s-1$.
\end{example}

There is a natural way of forming graphs associated with disjoint types of 
width $n$ in which the vertices are elements of $[\mathrm{Ord}]^n$. We will 
especially be interested in such graphs on the vertex set $[\omega_1]^n$:

\begin{definition}
    Suppose that $t$ is a disjoint type of width $n$. Then $G(t) = ([\omega_1]^n, E(t))$ 
    denotes the graph with vertex set $[\omega_1]^n$ such that, for all $a,b \in 
    [\omega_1]^n$, we put $\{a,b\} \in E(t)$ if and only if $\type(a,b) \in 
    \{t, \bar{t}\}$. Given natural numbers $1 \leq s < n$, we will denote 
    $G(t^n_s)$ by $\mathsf{S}^n_s$. Such graphs are often referred to as \emph{generalized 
    Specker graphs} (the particular graph $\mathsf{S}^3_1$ is sometimes referred to as 
    \emph{the} Specker graph).
\end{definition}

The following theorem is due to Erd\H{o}s and Hajnal.

\begin{theorem}[\cite{erdoshajnal1966}] \label{erdos_hajnal_thm}
  Suppose that $1 \leq s < n$ are natural numbers.
  \begin{enumerate}
      \item $\chi(\mathsf{S}^n_s) = \aleph_1$;
      \item if $n \geq 2s^2 + 1$, then $\mathsf{S}^n_s$ contains no odd cycles of length 
      $2s+1$ or shorter.
  \end{enumerate}
\end{theorem}

\begin{remark}
    A proof of Theorem \ref{erdos_hajnal_thm} is not given in \cite{erdoshajnal1966}. 
    A proof of item (1) can be found in multiple sources, and will also follow from 
    arguments in this paper. A proof of item (2) can be found in \cite{lambiehanson2020note}; 
    we note that the proof given there requires $n > 2s^2 + 3s + 1$, but the precise minimal 
    value for this $n$ will not be important for the results here.
\end{remark}

\begin{definition}
    A \emph{$C$-sequence} is a sequence of the form $\vec{C} = 
    \langle C_\alpha \mid \alpha \in \lim(\omega_1) \rangle$ such that, 
    for all $\alpha \in \lim(\omega_1)$, $C_\alpha$ is a cofinal subset 
    of $\alpha$ and $\ot(C_\alpha) = \omega$.
\end{definition}

\begin{remark}
    In some sources, $C$-sequences are also required to be defined at 
    successor ordinals, with $C_{\alpha + 1} = \{\alpha\}$ for all $\alpha 
    < \omega_1$, but this will not be relevant for us here. 
\end{remark}

We are now ready to introduce the notion of 
\emph{disjoint type guessing}, which will 
play an instrumental role in verifying that the HM graphs we construct later 
in the paper have uncountable chromatic number.

\begin{definition}
    Suppose that $\vec{C} = \langle C_\alpha \mid \alpha \in \lim(\omega_1) \rangle$ is a 
    $C$-sequence and $\vec{t} = \langle t_k \mid k < \omega \rangle$ is a sequence of 
    disjoint types. We say that $\vec{C}$ is a \emph{$\vec{t}$-guessing sequence} if, for 
    every function $f:\omega_1 \rightarrow \omega$, there are $\alpha < \beta$ in 
    $\lim(\omega_1)$ and $k < \omega$ such that
    \begin{enumerate}
        \item $f(\alpha) = f(\beta) = k$; and
        \item letting $n := \width(t_k)$, we have 
        \[
          \type(C_\alpha[n],C_\beta[n]) \in \{t_k, \bar{t}_k\}
        \]
        (in particular, $C_\alpha[n]$ and $C_\beta[n]$ are disjoint).
    \end{enumerate}
    We say that $\vec{C}$ is a \emph{strong} $\vec{t}$-guessing sequence if, 
    for every $f:\omega_1 \rightarrow \omega$, there is $\beta \in \lim(\omega_1)$ and 
    $k = f(\beta)$ such that the set of $\alpha \in \lim(\beta)$ for which $\alpha$ and 
    $\beta$ satisfy (1) and (2) above is unbounded in $\beta$.
    
    We say that $\emph{(strong) disjoint type guessing}$ ($\mathsf{(s)DTG}$) holds if, for every 
    sequence $\vec{t} = \langle t_k \mid k < \omega \rangle$ of disjoint types, 
    there exists a (strong) $\vec{t}$-guessing sequence.
\end{definition}

We will see at the end of this section that $\mathsf{sDTG}$ follows from club 
guessing (and hence also, \emph{a fortiori}, from $\diamondsuit$), 
and we will see in Section \ref{pfa_sect} that $\mathsf{DTG}$ is not a theorem of 
$\mathsf{ZFC}$. Let us show now, though, that $\mathsf{sDTG}$ \emph{is} a theorem 
of $\mathsf{ZFC}$ when restricted to a certain nice class of sequences of types.
Let us say that a sequence $\vec{t} = \langle t_k \mid k < \omega \rangle$ of disjoint 
types has \emph{bounded width} if $\sup\{\width(t_k) \mid k < \omega\} < \omega$. 

\begin{theorem} \label{zfc_guessing_thm}
    There is a $C$-sequence $\vec{C} = \langle C_\alpha \mid \alpha \in \lim(\omega_1) \rangle$
    such that, for every sequence $\vec{t} = \langle t_k \mid k < \omega 
    \rangle$ of disjoint types with bounded width, $\vec{C}$ is a strong 
    $\vec{t}$-guessing sequence. In fact, the following formally stronger 
    statement is true: for every $f : \omega_1 \rightarrow \omega$, there 
    are $\beta \in \lim(\omega_1)$ and $k < \omega$ such that $f(\beta) = 
    k$ and, for every $\eta < \beta$, there is $\alpha \in \lim(\beta)$ 
    such that, letting $n := \width(t_k)$, we have
    \begin{itemize}
        \item $f(\alpha) = k$;
        \item $\type(C_\alpha[n], C_\beta[n]) = t_k$;
        \item $C_\alpha(n) > \eta$.
    \end{itemize}
\end{theorem}

\begin{proof}
    Let $\vec{C}$ be any $C$-sequence such that, for every $b \in [\omega_1]^{<\omega}$, 
    there are stationarily many $\beta \in \lim(\omega_1)$ such that $b \sqsubseteq C_\beta$. 
    We claim that $\vec{C}$ is as desired. To this end, fix a sequence $\vec{t} = 
    \langle t_k \mid k < \omega \rangle$ of disjoint types such that $n^* := 
    \sup\{\width(t_k) \mid k < \omega\} < \omega$, and fix a function $f:\omega_1 
    \rightarrow \omega$. Let $\theta$ be a sufficiently 
    large regular cardinal, and let $\langle M_i \mid i \leq n^* \rangle$ be a $\in$-increasing 
    sequence of countable elementary substructures of $(H(\theta), \in, \vec{C}, 
    \vec{t}, f)$. For each $i \leq n^*$, let $\delta_i := M_i \cap \omega_1$. By our 
    choice of $\vec{C}$, we can find a countable elementary substructure 
    $M^* \prec (H(\theta), \in, \vec{C}, \vec{t}, f)$ such that 
    $M_{n^*} \in M^*$ and, letting $\beta = M^* \cap \omega_1$, we have 
    $C_\beta(i) = \delta_i$ for all $i \leq n^*$. Let $k := f(\beta)$ and 
    $n := \width(t_k) \leq n^*$. Assume that $n > 1$; the case in which $n = 1$ is similar 
    but much easier.

    For a formula $\psi$, let us abbreviate by $\exists^\infty\varepsilon : \psi(\varepsilon)$ 
    the formula $\forall \zeta < \omegaone \exists \varepsilon < \omegaone : \zeta < \varepsilon 
    \land \psi(\varepsilon)$. Informally, this is saying that there are unboundedly 
    many $\varepsilon < \omega_1$ for which $\psi(\varepsilon)$ holds. 
    Given a $b \in [\omega_1]^{n+1}$ and an $\alpha < \omega_1$, 
    let $\varphi(b, \alpha)$ be the statement asserting that $\alpha \in \lim(\omega_1)$, 
    $f(\alpha) = k$, and $b \sqsubseteq C_\alpha$. Then $\varphi(\{\delta_i \mid i \leq n\}, \beta)$ 
    holds. In particular, for all $\zeta < \beta$, the statement $\exists \alpha < 
    \omega_1 : \zeta < \alpha \land \varphi(\{\delta_i \mid i \leq n\}, \alpha)$ holds, 
    as witnessed by $\beta$. By elementarity, this statement is true in $M^*$ as well. 
    Since $\beta = M^* \cap \omega_1$, this in fact implies that
    \[
      M^* \models \exists^\infty \alpha : \varphi(\{\delta_i \mid i \leq n\}, \alpha),
    \]
    and hence, by another application of elementarity, this statement is true in $H(\theta)$.

    Now, for all $\zeta < \delta_n$, we have 
    \[
      H(\theta) \models \exists \gamma < \omega_1 : \zeta < \gamma \land (\exists^\infty 
      \alpha : \varphi\{\{\delta_i \mid i < n\} \cup \{\gamma\}, \alpha),
    \]
    as witnessed by $\gamma = \delta_{n}$. By elementarity, this statement is true in 
    $M_{n}$. Since $\delta_{n} = M_{n} \cap \omega_1$, this implies that
    \[
      M_{n} \models \exists^\infty \gamma \exists^\infty \alpha : 
      \varphi(\{\delta_i \mid i < n\} \cup \{\gamma\}, \alpha),
    \]
    so, again by another application of elementarity, this statement is true in $H(\theta)$ 
    as well. Continuing in this way, we see that the following statement holds in $H(\theta)$, 
    and hence also in $M^*$ and $M_i$ for every $i \leq n^*$:
    \[
      \exists^\infty \gamma_0 \exists^\infty \gamma_1 \ldots \exists^\infty \gamma_{n} 
      \exists^\infty \alpha : \varphi(\{\gamma_i \mid i \leq n\}, \alpha).
    \]

    Without loss of generality, assume that $t_k(2n-1) = 1$; the other case 
    is symmetric. For $\ell < 2$, let $e_\ell := t_k^{-1}\{\ell\} \in [2n]^n$, and
    define an auxiliary function $s:n \rightarrow n$ by letting $s(i)$ be the least 
    $j < n$ such that $e_0(i) < e_1(j)$. For notational convenience, let 
    $\delta_{-1} = \gamma^*_{-1} = -1$. We now recursively choose an increasing sequence of 
    ordinals $\langle \gamma^*_i \mid i < n \rangle$ such that, for all 
    $i < n$, we have
    \begin{itemize}
        \item $\delta_{s(i) - 1} < \gamma^*_i < \delta_{s(i)}$;
        \item $\exists^\infty \gamma_{i+1}\ldots \exists^\infty \gamma_{n} 
        \exists^\infty \alpha : \varphi(\{\gamma^*_j \mid j \leq i\} 
        \cup \{\gamma_j \mid i < j \leq n\}, \alpha)$.
    \end{itemize}
    The construction is straightforward: if $i < n$ and we have already 
    chosen $\langle \gamma^*_j \mid j < i \rangle$, then, by our recursion 
    hypothesis, we have 
    \begin{itemize}
        \item $\{\gamma^*_j \mid j < i \} \in M_{s(i)}$; and
        \item $M_{s(i)} \models \exists^\infty \gamma_i\ldots \exists^\infty 
        \gamma_{n} \exists^\infty \alpha : \varphi(\{\gamma^*_j \mid j \leq i\} 
        \cup \{\gamma_j \mid i < j \leq n\}, \alpha)$.
    \end{itemize}
    We can therefore choose a $\gamma^*_i$ witnessing the statement in the second 
    bullet point above satisfying $\max\{\delta_{s(i)-1}, \gamma^*_{i-1}\} < \gamma^*_i 
    < \delta_{s(i)}$. This choice continues to satisfy the recursion hypothesis, 
    so we can then continue to the next step in the construction. At the end, we have arranged so that
    \begin{itemize}
        \item $\type(\{\gamma^*_i \mid i < n\}, \{\delta_i \mid i < n\}) = t_k$; and
        \item $\exists^\infty \gamma_n \exists^\infty \alpha : 
        \varphi(\{\gamma^*_i \mid i < n\} \cup \{\gamma_n\}, \alpha)$.
    \end{itemize}
    These statements hold in $M^*$, and $M^* \cap \omega_1 = \beta$.
    Therefore, for every $\eta < \beta$, we can find $\gamma^*_n > \eta$ and
    $\alpha \in \lim(\beta)$ such that $f(\alpha) = k$ and $C_\alpha(i) = 
    \gamma^*_i$ for all $i \leq n$. Thus, $\beta$ witnesses the conclusion 
    of the theorem.
\end{proof}

Recall that \emph{club guessing} holds if there is a $C$-sequence 
$\vec{C} = \langle C_\alpha \mid \alpha \in \lim(\omega_1) \rangle$ 
such that, for every club $D$ in $\omega_1$, there is 
$\alpha \in \lim(\omega_1)$ for which $C_\alpha \subseteq D$. 
It is evident that club guessing follows from $\diamondsuit$.
We now show that the full $\mathsf{sDTG}$ follows from club guessing, and 
hence from $\diamondsuit$.

\begin{theorem}
    Suppose that $\vec{C} = \langle C_\alpha \mid \alpha \in \lim(\omega_1) \rangle$ is a witness to club guessing. Then $\vec{C}$ is a 
    strong $\vec{t}$-guessing sequence for every sequence $\vec{t} = \langle t_k \mid k 
    < \omega \rangle$ of disjoint types.
\end{theorem}

\begin{proof}
    Fix a sequence $\vec{t} = \langle t_k \mid k < \omega \rangle$ of disjoint types and 
    a function $f:\omega_1 \rightarrow \omega$. Let $\theta$ be a sufficiently 
    large regular cardinal, and let $\langle N_\eta 
    \mid \eta < \omega_1 \rangle$ be a $\in$-increasing, continuous sequence of 
    countable elementary substructures of $(H(\theta), \in, \vec{C}, \vec{t}, 
    f)$. Let $D := \{\eta < \omega_1 \mid \eta = N_\eta \cap \omega_1\}$. 
    Then $D$ is a club in $\omega_1$, so we can find $\beta \in \lim(\omega_1)$ such 
    that $C_\beta \subseteq D$. Let $k := f(\beta)$, and let $n := \width(t_k)$. 
    For $i \leq n$, let $\delta_i := C_\beta(i)$ and $M_i := N_{\delta_i}$. 
    Since $\delta_i \in D$, we also have $\delta_i = M_i \cap \omega_1$. 
    Let $M^* := N_\beta$. We are now in precisely the situation from the proof of 
    Theorem \ref{zfc_guessing_thm}, so we can repeat the arguments from that proof 
    to conclude that there are unboundedly many $\alpha \in \lim(\beta)$ such 
    that $f(\alpha) = k$ and $\type(C_\alpha[n],C_\beta[n]) = t_k$. Therefore, 
    $\beta$ witnesses this instance of strong $\vec{t}$-guessing.
\end{proof}

\section{Triangle-free Hajnal--Máté graph from a single Cohen real}
\label{main_result_sec}

We begin this section with a generalized definition of Hajnal--Máté graphs.
Recall that a \emph{tree} is a partial order $(T, <_T)$ 
such that, for every $t \in T$, the set $\mathrm{pred}_T(t) := \{s \in 
T \mid s <_T t\}$ is well-ordered by $<_T$. We will often simply use 
$T$ to denote the tree $(T, <_T)$. For an ordinal $\beta$, the 
\emph{$\beta$-th level of $T$}, denoted $T_\beta$, is the set of all 
$t \in T$ such that $\ot(\mathrm{pred}_T(t), <_T) = \beta$. 
We also let $T_{<\beta} = \bigcup_{\alpha < \beta} T_\alpha$. The 
\emph{height} of $T$ is the least ordinal $\beta$ such 
that $T_\beta = \emptyset$. The \emph{comparability graph} of $T$ 
is the graph $(T,E)$ such that, for all $s,t \in T$, we have 
$\{s,t\} \in E$ if and only if either $s <_T t$ or $t <_T s$. 

\begin{definition} \label{hm_def}
    Let $T$ be a tree of height $\omegaone$. A graph $G = (T, E)$ is called a \emph{$T$-Hajnal--Máté graph} (or \emph{$T$-HM graph}) if it is a subgraph of the comparability graph of $T$, the chromatic number of $G$ is uncountable, and for all $\alpha<\beta<\omegaone$ and $t \in T_\beta$ the set $\{s \in T_{<\alpha} \mid \{s, t\} \in E \}$ is finite.
\end{definition}
\begin{remark}
If $T$ is the tree $(\omegaone, \in)$ we omit the parameter $T$. Thus an $(\omegaone, \in)$-Hajnal--Máté graph will be called a Hajnal--Máté graph, or HM graph for short. Note that the vertex set of these graphs is always a tree, hence ordered.

We also remark that there is some inconsistency in terminology in the extant literature, 
and in some other works Hajnal--M\'{a}t\'{e} graphs 
are not required by definition to be uncountably chromatic. Since the HM graphs of interest are those 
that are uncountably chromatic, we adopt here the convention that HM graphs are 
necessarily uncountably chromatic.
\end{remark}

In this section we give a positive answer to Soukup's question asking whether adding a single Cohen real forces the existence of a triangle-free HM graph. First, we need a few preliminary definitions and lemmas.

In our construction of the HM graph, we will ensure that there is a graph homomorphism from the constructed graph to a suitable $\mathsf{S}^n_s$. The following lemma, together with Theorem \ref{erdos_hajnal_thm}, ensures that the graph will have no short odd cycles. The proof of the following lemma is in \cite[Proposition 2.8.]{lambiehanson2020b}.

\begin{lemma} \label{nooddcyclelemma}
    Suppose $G = (X_G, E_G)$ and $H = (X_H, E_H)$ are graphs and there is a map $f: X_G \to X_H$ such that, for all $x,y \in X_G$, if $\{x,y\} \in E_G$, then $\{f(x), f(y)\} \in E_H$. If $s \in \omega$ and $H$ has no odd cycles of length $2s+1$ or shorter, then $G$ has no odd cycles of length $2s+1$ or shorter.
\end{lemma}

In the HM graph we construct, we will also forbid cycles formed by two monotone paths.

\begin{definition}
    Suppose $G$ is a graph on $\omegaone$. A cycle $\langle x_0, \ldots, x_{n-1}, x_0\rangle$ is called \emph{special} if there is an $r < n$ such that $x_i > x_{i+1}$ for $i < r$ and $x_i < x_{i+1}$ for $r \le i$. A graph is said to be \emph{special cycle--free} if it contains no special cycles.
\end{definition}

Recall that $H_{\omega, \omega+2}$ is the graph with the vertex set made up of disjoint sets $\{x_i \mid i < \omega\}$ and $\{y_i \mid i < \omega+2\}$, where for each $i < \omega$ and $j < \omega+2$ such that $i\le j$ the vertex $x_i$ is connected to $y_j$. The graph $H_{\omega, \omega+1}$ is the same graph with the vertex $y_{\omega+1}$ omitted. Hajnal and Komjáth \cite[Theorem 1]{hajnalkomjath1984} showed that $H_{\omega, \omega+1}$ is a subgraph of every uncountably chromatic graph.

Note that if $G$ is special cycle--free, it is triangle-free. Additionally, the following lemma says that being special cycle--free also forbids $H_{\omega, \omega+2}$ in HM graphs \cite[Theorem 3]{hajnalkomjath1984}.
In what follows, we say that a tree $T$ of height $\omega_1$ 
\emph{does not split on limit levels} if, for all $\beta \in 
\lim(\omega_1)$ and all $s,t \in T_\beta$, if $\mathrm{pred}_T(s) = 
\mathrm{pred}_T(t)$, then $s = t$.

\begin{lemma}[Hajnal and Komjáth] \label{nokomegaomegainmonotonefree}
    Suppose $T$ is a tree of height $\omegaone$ that does not split on limit levels, and $G$ is a special cycle--free $T$-HM graph. Then $G$ is triangle-free and $H_{\omega, \omega+2}$ is not a subgraph of $G$.
\end{lemma}
\begin{proof}
    Clearly, $G$ has no triangles. Suppose $\{x_i \mid i < \omega\}$ and $\{y_i \mid i < \omega+2\}$ form an $H_{\omega, \omega+2}$ subgraph in $G$. For each $i < \omega$, let $\alpha(i) < \omega_1$ be such 
    that $x_i \in T_{\alpha(i)}$. First, note that $\{x_i \mid i < \omega\}$ is linearly ordered 
    by $<_T$. Otherwise, there would be $<_T$-incomparable $x_i$ and $x_j$ with infinitely many common neighbors. However, as $T$ does not split on limit levels, it must be the case that $\alpha(i) \neq \alpha(j)$. This contradicts the definition of a $T$-HM graph.
    
    Next, note that $\{x_i \mid i < \omega\} \subseteq \mathrm{N}(y_\omega) \cap \mathrm{N}(y_{\omega+1})$. The nodes $y_\omega$ and $y_{\omega+1}$ cannot lie in the same level of $T$ as the tree does not split on limit levels, and they share infinitely many common neighbors. Suppose that $y_\omega \in T_\gamma$, $y_{\omega+1} \in T_\delta$ and $\gamma < \delta$ (the other case is symmetric). Then, as $G$ is $T$-HM it follows that only finitely many elements of $\{x_i \mid i < \omega\}$ lie below level $\gamma$. Choose $k < \omega$ such that 
    $\gamma < \alpha(k) < \alpha(k+1)$.  By the arguments of the previous 
    paragraph, $x_k$ and $x_{k+1}$ have infinitely many common neighbors 
    whose levels are above $\alpha(k+1)$; choose one such neighbor, and 
    call it $z$. Then $\langle y_\omega, x_k, z, x_{k+1},y\rangle$ is 
    a special cycle in $G$, which is a contradiction.
\end{proof}
\begin{remark}
    Recall that $K_{\omega,\omega}$ denotes the complete bipartite graph 
    with countably infinitely many vertices on each side. Since 
    $H_{\omega, \omega+2}$ is a subgraph of $K_{\omega,\omega}$, 
    it follows that $K_{\omega, \omega}$ is also forbidden in special cycle-free $T$-HM graphs.
\end{remark}

Lemma \ref{nokomegaomegainmonotonefree} will be applicable to the graphs constructed in Theorems \ref{cohenandhmgraph} and \ref{thmgraphinzfc}. In particular, those graphs will not have $H_{\omega, \omega+2}$ as a subgraph.
We are now ready to answer Soukup's question.

\begin{theorem} \label{cohenandhmgraph}
    Adding a Cohen real forces that for each $s\in\omega$ there is a special cycle--free Hajnal--Máté graph without odd cycles of length at most $2s+1$.
\end{theorem}
\begin{proof}
    In the ground model, fix bijections $e_\beta: \omega \to \beta$ for each infinite countable ordinal $\beta$ and let $n\in\omega$ be such that $\mathsf{S}^n_s$ (the generalization of the Specker graph) has no odd cycles of length $2s+1$ or shorter. Recall that $t^n_s$ is the type associated with this graph. Fix a $C$-sequence $\vec{C}$ satisfying Theorem \ref{zfc_guessing_thm}.

    Let $r: \omega \to \omega$ be the Cohen real in the extension. We will define a graph $G$ on $\omegaone$ by specifying the set of smaller neighbors $\mathrm{N}^<(\beta)$ for each $\beta<\omegaone$. We proceed by induction on $\beta$. Given $\delta<\beta$ let us say that $\beta$ is \emph{$\delta$-covered} if there is a monotone decreasing path from $\beta$ to an ordinal $\alpha$ such that $\alpha \le \delta$. Our construction will ensure that for each $\beta \in \lim(\omegaone)$, we have that $\beta$ is not $C_\beta(n)$-covered and, if $\beta < \omega_1$ is a successor ordinal, 
    then $\mathrm{N}^<(\beta) = \emptyset$.

    Suppose we have constructed $G$ up to some $\beta$, i.e., $\mathrm{N}^<(\alpha)$ is defined for each $\alpha<\beta$. If $\beta$ is a successor ordinal, it will have no neighbors below $\beta$. Assume $\beta$ is limit. We also inductively assume that for each $\alpha \in \lim(\beta)$ we have that $\alpha$ is not $C_\alpha(n)$-covered. By induction on $k\in\omega$ we construct a set $K_\beta \in [\omega]^{\le\omega}$ and limit ordinals $\{\beta_k \mid k \in K_\beta\}$ such that:
    \begin{enumerate}
        \item for all $k \in K_\beta$, we have $\beta_k = e_\beta(r(k))$;
        \item for all $k \in K_\beta$, we have $C_{\beta_k}(n) > \max (\{\beta_i \mid i \in K_\beta \cap k\} \cup \{C_\beta(n)\})$;
        \item for all $k \in K_\beta$, we have $\beta_k > \max (\{\beta_i \mid i \in K_\beta \cap k\} \cup \{C_\beta(k)\})$;
        \item for all $k \in K_\beta$, we have $\type(C_\alpha[n],C_\beta[n]) \in \{t^n_s, \overline{t^n_s}\}$, i.e. $\{C_\beta[n], C_{\beta_k}[n]\}$ is an edge in $\mathsf{S}^n_s$.
    \end{enumerate}
    At stage $k$ of the construction, we consider the ordinal $e_\beta(r(k))$ and let $k\not\in K_\beta$ unless it satisfies all of the conditions above. If it \emph{does} satisfy all of the 
    conditions, we let $\beta_k := e_\beta(r(k))$ and put $k$ into $K_\beta$. Finally we let $\mathrm{N}^<(\beta)$ be $\{\beta_k \mid k \in K_\beta\}$. Note that $\beta$ is not $C_\beta(n)$-covered: for each $\alpha \in \mathrm{N}^<(\beta)$ we have $C_{\alpha}(n) > C_{\beta}(n)$. The induction hypothesis now gives that there is no monotone decreasing path from any such $\alpha$ to an ordinal below $C_{\beta}(n)$. In particular, there is no monotone path from $\beta$ to an ordinal less than or equal to $C_{\beta}(n)$.

    By the third condition we obtain that for each $\alpha<\beta$ we have $|\mathrm{N}(\beta) \cap \alpha| < \omega$. The second condition ensures that if $\alpha < \alpha' < \beta$ and $\alpha, \alpha'$ are both elements of $\mathrm{N}^<(\beta)$, then $\alpha'$ is not $\alpha$-covered, in particular $\alpha, \alpha', \beta$ cannot be the three topmost elements of a special cycle, hence $G$ is special cycle--free. Last but not least, the fourth condition ensures that $G$ can have no odd cycles of length $2s+1$ or less. To see this, consider the map which takes $\beta$ to $C_{\beta}[n]$ (as a vertex in $\mathsf{S}^n_s$). 
    The fourth condition ensures that this is a graph homomorphism from 
    $G$ to $\mathsf{S}^n_s$, so, by Lemma \ref{nooddcyclelemma} the graph $G$ can have no odd cycle of length $2s+1$ or shorter.

    It remains to prove that $G$ is uncountably chromatic.
    In the ground model, let $\dot{G}$ be a name for the graph 
    constructed above, and let $\dot{c}$ be a name such that $\cohengen{\omega}{1} \Vdash \dot{c}: \omegaone \to \omega$, i.e., 
    $\dot{c}$ is a name for a coloring of $\dot{G}$. We must show that it is forced that $\dot{c}$ is not a proper coloring of $\dot{G}$. To this end, let $p$ be an arbitrary condition in $\cohengen{\omega}{1}$. 
    Consider the coloring $d: \omegaone \to \omega \times \cohengen{\omega}{1}$ such that $d(\alpha) = (k, q)$ if and only if 
    \begin{itemize}
        \item $q$ is the least condition (in some well-ordering of $\cohengen{\omega}{1}$) extending $p$ such that $q$ decides the value of $\dot{c}(\alpha)$; and
        \item $q \Vdash \dot{c}(\alpha) = k$.
    \end{itemize}
    Since $\omega \times \cohengen{\omega}{1}$ is countable, we can 
    apply Theorem \ref{zfc_guessing_thm} to the constant sequence
    of disjoint types taking value $t^n_s$ to obtain a $\beta \in \lim(\omegaone)$, a $k < \omega$, and a condition $q \leq p$ such that for each $\eta < \beta$ there is an $\alpha \in \lim(\beta)$ such that 
    \begin{itemize}
        \item $C_\alpha(n) > \eta$;
        \item $q \Vdash \dot{c}(\alpha) = k = \dot{c}(\beta)$
        \item $\{C_\beta[n], C_{\alpha}[n]\}$ is an edge in $\mathsf{S}^n_s$.
    \end{itemize}
    
    Put $m := |q|$ and note that $q$ decides the first $m$ candidates for neighbors of $\beta$ as $q \Vdash e_\beta(\dot{r}(i)) = e_\beta(q(i))$ for all $i<m$. Put $\lambda := \max\{e_\beta(q(i)) \mid i < m\}$, and note that $\lambda<\beta$. Choose $\alpha^* \in \lim(\beta)$ so that 
    \begin{itemize}
        \item $C_{\alpha^*}(n) > \max\{\lambda, C_\beta(n), C_\beta(m)\}$;
        \item $q \Vdash \dot{c}(\alpha^*) = k$
        \item $\{C_\beta[n], C_{\alpha^*}[n]\}$ is an edge in $\mathsf{S}^n_s$.
    \end{itemize}
    Define $q^* := q \cup \{(m, e_\beta^{-1}(\alpha^*))\}$.

    By construction we have that $q^* \Vdash \dot{c}(\alpha^*)=\dot{c}(\beta) = k$. Moreover, we claim that $q^* \Vdash \alpha^* \in \mathrm{N}^<(\beta)$. Indeed, suppose we are in a generic extension by a Cohen real $r$ and $q^* \subseteq r$. At stage $\beta$ in the construction consider the $m$-th step. The ordinal $e_\beta(r(m))$ at that point is exactly $\alpha^*$, as $q^* \Vdash e_\beta(\dot{r}(m))=\alpha^*$. Note that our choice of $\alpha^*$ ensures that $\alpha^*$ satisfies all the conditions for it being a neighbor of $\beta$ so $\beta_m := \alpha^*$, hence $\alpha^* \in \mathrm{N}^<(\beta)$. In particular, $q^*$ forces that $\dot{c}$ is 
    not a proper coloring of $\dot{G}$. We have shown that the set of conditions forcing that $\dot{c}$ is not a name for a proper coloring is dense. Thus $G$ cannot be countably chromatic.
\end{proof}

Using the same technique, we can construct a simple $\zfc$ example of a $T$-HM graph with no special cycles and no short odd cycles. The first such graph was constructed by Hajnal and Komjáth \cite[Theorem 2.]{hajnalkomjath1984}, but the construction used $\ch$. A $\zfc$ example 
of a triangle-free and special cycle-free $T$-HM graph is due to Soukup and uses a tree of the form $\{t \subseteq S \mid t \text{ is closed}\}$ where $S \subseteq \omegaone$ is stationary, co-stationary, and the tree order is end extension \cite[Theorem~5.5]{soukup2015}. The tree we will consider is ${}^{<\omegaone}\omega$, ordered by end-extension. Our theorem extends Komjáth and Shelah's result \cite[Theorem 10]{komjathshelah1988} by excluding all special cycles. Note that ${}^{<\omegaone}\omega$ does not split on limit levels.

\begin{theorem} \label{thmgraphinzfc}
    For each $s \in \omega$, there is a special cycle--free ${}^{<\omegaone}\omega$-HM graph without odd cycles of length at most $2s+1$.
\end{theorem}
\begin{proof}
     The main ideas of the proof are the same as in Theorem \ref{cohenandhmgraph}. We point out the differences. In the beginning, we fix the same objects: the graph $\mathsf{S}^n_s$ and the $C$-sequence $\vec{C}$ from Theorem \ref{zfc_guessing_thm}. For $\delta < \beta < \omega_1$ and $f \in {}^\beta\omega$ we say that \emph{$f$ is $\delta$-covered} if there is a monotone decreasing path from $f$ to $f\restriction\gamma$ for some $\gamma\le\delta$. We will construct the desired graph $G = ({^{<\omega_1}}\omega, E)$. By $\mathrm{N}^<(f)$ we will denote the set $\{f\restriction\gamma \mid \gamma < \beta \land \{f\restriction\gamma, f\} \in E\}$.

    The construction proceeds by specifying $\mathrm{N}^<(f)$ by recursion on the levels of the tree, and the recursion hypothesis is that if $f \in {}^\beta\omega$ then $f$ is not $C_\beta(n)$-covered. If $\beta$ is a successor ordinal and $f \in {^\beta}\omega$, then $\mathrm{N}^<(f)$ will be empty. If $\beta$ is limit, the construction for each $f \in {}^\beta\omega$ is identical. Let us fix an $f \in {}^\beta\omega$. We use induction on $k\in\omega$ to construct a set $K_f \in [\omega]^{\le\omega}$ and limit ordinals $\{\beta^f_k \mid k \in K_f\}$ such that:
    \begin{enumerate}
        \item for all $k \in K_f$, we have $f(\beta^f_k) = k$;
        \item for all $k \in K_f$, we have $C_{\beta^f_k}(n) > \max (\{\beta^f_i \mid i \in K_f \cap k\} \cup \{C_\beta(n)\})$;
        \item for all $k \in K_f$, we have $\beta^f_k > \max (\{\beta^f_i \mid i \in K_\beta \cap k\} \cup \{C_\beta(k)\})$;
        \item for all $k \in K_f$, we have $\type(C_\beta[n],C_{\beta^f_k}[n]) \in \{t^n_s, \overline{t^n_s}\}$, i.e. $\{C_\beta[n], C_{\beta^f_k}[n]\}$ is an edge in $\mathsf{S}^n_s$.
    \end{enumerate}
    At stage $k$ if there is some $\alpha$ satisfying the required properties, we let $\beta^f_k$ be the minimal such $\alpha$ and put $k\in K_f$. Otherwise $\beta^f_k$ is undefined and $k \not\in K_f$. Lastly we put $\mathrm{N}^<(f) := \{f\restriction\beta^f_k \mid k \in K_f\}$.

    The fact that this graph has no special cycles and no short odd cycles is proved analogously as in Theorem \ref{cohenandhmgraph}. We only comment on the chromaticity. Fix a coloring $c: {}^{<\omegaone}\omega \to \omega$ and define $g:\omegaone \to \omega$ by recursion as $g(\beta) = c(g\restriction\beta)$.

    As in the previous proof, by Theorem \ref{zfc_guessing_thm} the sequence $\vec{C}$ strongly guesses the constant sequence $\langle t^n_s \mid i < \omega \rangle$ with respect to the function $g$. Let $\beta < \omegaone$ and $k < \omega$ be witnesses for the strong guessing. Put $\lambda := \max\{\beta^{g\restriction\beta}_i \mid i < k\}$. Choose $\alpha < \beta$ so that $g(\alpha) = k$, $\{C_\beta[n], C_{\alpha}[n]\}$ is an edge in $\mathsf{S}^n_s$ and $C_{\alpha}(n) > \max\{\lambda, C_\beta(n), C_\beta(k)\}$.
    
    The existence of such an $\alpha$ implies that there is some minimal example at stage $k$ in the construction of the set $\mathrm{N}^<(g\restriction\beta)$, so the element $\beta^{g\restriction\beta}_k$ was defined. Hence $c(g\restriction\beta) = g(\beta) = k = g(\beta^{g\restriction\beta}_k) = c(g\restriction\beta^{g\restriction\beta}_k)$, so $c$ is not a proper coloring.
\end{proof}

\section{Growth rates of chromatic numbers} \label{growth_rate_sec}

By the De Bruijn-Erd\H{o}s compactness theorem \cite{debruijn1951}, if $G$ is a graph of infinite chromatic number, 
then, for every $k < \omega$, there is a finite subgraph of $G$ with chromatic number $k$. One 
can then define a (strictly increasing) function $f_G:\omega \rightarrow \omega$ by letting, 
for all $k < \omega$, $f_G(k)$ be the least number of vertices in a subgraph of $G$ with 
chromatic number $k$. In \cite{erdos1982}, Erd\H{o}s, Hajnal, and Szemer\'{e}di asked whether it is the case 
that, for every function $f:\omega \rightarrow \omega$, there is a graph $G$ of 
uncountable chromatic number such that $f_G$ grows more quickly than $f$. In 
\cite{lambiehanson2020b}, the first author answered this question positively, in fact 
producing, for each function $f$, a graph $G$ of chromatic number $\omega_1$ such that 
$f_G(k) > f(k)$ for all $3 \leq k < \omega$. It was additionally shown there that, if 
$\diamondsuit$ holds, then this graph $G$ can be taken to be an HM graph. We show here 
that, if $\mathsf{DTG}$ holds in the ground model, then the positive answer to the 
question of Erd\H{o}s, Hajnal, and Szemer\'{e}di is witnessed by HM graphs in the 
extension by a single Cohen real.

\begin{theorem} \label{growth_rate_thm}
    Suppose that $\mathsf{DTG}$ holds. Then, in the 
    forcing extension by a single Cohen real, the following statement is true: For 
    every function $f:\omega \rightarrow \omega$, there is an HM graph $G$ such that 
    $f_G(k) > f(k)$ for all $3 \leq k < \omega$.
\end{theorem}

\begin{proof}
    Let $\P = \cohengen{\omega}{1}$. We will show that the following 
    statement holds in the forcing extension by $\P$: for every $f:\omega \rightarrow \omega$, there is an 
    HM graph $G = (\omega_1, E)$ such that, for every $k < \omega$, 
    if $H$ is a subgraph of $G$ with at most $f(k)$ vertices, then 
    $\chi(H) \leq 2^{k+1}$. This clearly suffices for the theorem.
    
    Let $\dot{r}$ be the canonical $\P$-name for the 
    Cohen real, and let $\dot{f}$ be a $\P$-name for an arbitrary function from 
    $\omega$ to $\omega$. For each $p \in \P$, fix an extension $p' \leq p$ of 
    minimal length such that $p'$ decides the value of $\dot{f}(j)$ for 
    all $j \leq |p|$; for each such $j$, let $s_{p,j}$ be the least $0 < s < \omega$ 
    such that $p' \Vdash \dot{f}(j) \leq 2s + 1$. Let $n_{p,j} := 2s_{p,j}^2 + 1$, 
    and let $n^*_p := \sum_{j\leq |p|} n_{p,j}$. 
    For each $p \in \P$ and $m < \omega$, let $t_{p,m}$ be the concatenation 
    $t^{n_{p,0}}_{s_{p,0}}{}^\frown t^{n_{p,1}}_{s_{p,1}}{}^\frown 
    \ldots{}^\frown t^{n_{p,|p|}}_{s_{p,|p|}}$ of length $n^*_p$ 
    (note that this does not 
    depend on the value of $m$). Using our type-guessing assumption, fix a $C$-sequence $\langle C_\alpha \mid 
    \alpha \in \lim(\omega_1)\rangle$ such that, for every function $g:\omega_1 \rightarrow 
    \P \times \omega$, there are distinct $\alpha, \beta \in \lim(\omega_1)$ and 
    $(p,m) \in \P \times \omega$ such that $g(\alpha) = g(\beta) = (p,m)$ and
    $\type(C_\alpha[n^*_p], C_\beta[n^*_p]) = t_{p,m}$. Also, for each infinite $\beta 
    < \omega_1$, fix a bijection $e_\beta : \omega \rightarrow \beta$.

    We temporarily move to the extension by $\P$ and describe how to construct in 
    that model the desired HM graph $G = (\omega_1, E)$. Let $r:\omega \rightarrow \omega$ be the 
    realization of $\dot{r}$ in the extension; note that the generic filter 
    is precisely $\{r \restriction \ell \mid \ell < \omega\}$. For each 
    $k < \omega$, let $s_k$ be the least $0 < s < \omega$ such that 
    $f(k) \leq 2s + 1$, let $n_k = 2s_k^2 + 1$, and let 
    $\ell_k < \omega$ be the least $\ell \geq k$ such that 
    $r \restriction \ell$ decides the value of $\dot{f}(j)$ for all $j \leq k$.
    Let $\{I_k \mid k < \omega\}$ be the partition of $\omega$ into adjacent intervals 
    such that $|I_k| = n_k$. More precisely, $I_0 := n_0$ and, if 
    $I_k$ has been defined, let $m_k = \max\{I_k\}$ and let 
    $I_{k+1} := \{m_k + 1 + j \mid j < n_{k+1}\}$.
    
    As in the proof of Theorem~\ref{cohenandhmgraph}, 
    we specify, for each $\beta < \omega_1$, the set 
    $N^<(\beta)$ of smaller neighbors of $\beta$. If $\beta$ is a 
    successor ordinal, we let $N^<(\beta) = \emptyset$. If 
    $\beta$ is a limit ordinal, we let $N^<(\beta)$ 
    be the set of all $\alpha \in \lim(\beta)$ for which there exists $k < \omega$ such that
    \begin{enumerate}
        \item $\alpha \geq C_\beta(k)$;
        \item $e_\beta(r(\ell_k)) = \alpha$; and
        \item for all $j \leq k$, we have $\type(C_\alpha[I_j], C_\beta[I_j]) 
        \in \{t^{n_j}_{s_j}, \bar{t}^{n_j}_{s_j}\}$.
    \end{enumerate}
    Requirements (1) and (2) above ensure that, for all $\alpha < 
    \beta < \omega_1$, the set $N^{<}(\beta) \cap \alpha$ 
    is finite. We next argue that the finite subgraphs of $G$ 
    behave as desired. This argument is essentially the same 
    as the analogous one in the proof of 
    \cite[Theorem B]{lambiehanson2020b}. First, for all $\alpha < 
    \beta < \omega_1$ with $\{\alpha,\beta\} \in E$, let $k_{\alpha, 
    \beta}$ be the minimal $k < \omega$ witnessing that $\alpha$ and 
    $\beta$ satisfy requirements (1)--(3) above. For each $k < \omega$, 
    let $E_k := \{\{\alpha, \beta\} \in E \mid k_{\alpha,\beta} = k\}$ and 
    $E_{\geq k} := \{\{\alpha, \beta\} \in E \mid k_{\alpha,\beta} \geq k\}$. 
    Given a subgraph $H = (X_H, E_H)$ of $G$ and a $k < \omega$, let 
    $H_k := (X_H, E_H \cap E_k)$ and $H_{\geq k} := (X_H, E_H \cap E_{\geq k})$.

    Now fix a $k < \omega$ and a nonempty subgraph $H$ of $G$ with at most 
    $f(k)$ vertices. Note that the graphs
    \[
      H_0, H_1, \ldots, H_{k-1}, H_{\geq k}
    \]
    form an edge-partition of $H$ and therefore (cf.\ 
    \cite[Proposition 2.6]{lambiehanson2020b}), we have $\chi(H) \leq 
    \chi(H_{\geq k}) \cdot \prod_{j < k} \chi(H_j)$ . 
    For each $j < k$ and $\beta < \omega_1$, there is at most one 
    $\alpha < \beta$ such that $\{\alpha,\beta\} \in E_j$. It follows that 
    $H_j$ has no cycles, so $\chi(H_j) \leq 2$. 

    Next, by requirement (3) above in the definition of $G$, the map 
    $\alpha \mapsto C_\alpha[I_k]$ induces a graph homomorphism from 
    $H_{\geq k}$ to $\mathsf{S}^{n_k}_{s_k}$. Therefore, by 
    Theorem~\ref{erdos_hajnal_thm} and Lemma~\ref{nooddcyclelemma}, 
    $H_{\geq k}$ has no odd cycles of length $2s_k+1$ or shorter. 
    But $H$ itself has at most $f(k) \leq 2s_k+1$ vertices, so 
    $H_{\geq k}$ has no odd cycles, and hence again $\chi(H_{\geq k}) 
    \leq 2$. It follows that $\chi(H) \leq 2 \cdot 2^k = 2^{k+1}$, as 
    desired.

    It remains to show that $G$ is uncountably chromatic. To this end, 
    move back to the ground model, and let $\dot{G}$ be a $\P$-name for 
    the graph constructed in the forcing extension as described above.
    Suppose that $\dot{c}$ is a $\P$-name for a function from $\omega_1$ 
    to $\omega$ and that $p \in \P$. We will find a condition $s \leq p$ forcing 
    that $\dot{c}$ is not a proper coloring of $\dot{G}$.

    First, for each $\alpha < \omega_1$, find a condition $q_\alpha \leq p$ 
    deciding the value of $\dot{c}(\alpha)$, say as $m_\alpha < \omega$.
    By the properties of $\vec{C}$, we can find a pair $(q,m) \in \P \times \omega$ 
    and distinct $\alpha, \beta \in \lim(\omega_1)$ such that $(q_\alpha, m_\alpha) = 
    (q_\beta, m_\beta) = (q,m)$ and $\type(C_\alpha[n^*_q], C_\beta[n^*_q]) = 
    t_{q,m}$. Without loss of generality, assume that $\alpha < \beta$; the 
    argument in the other case is symmetric.

    Let $k := |q|$. Recall that $q'$ is an extension of $q$ of 
    minimal length deciding the value of $\dot{f}(j)$ for all 
    $j \leq k$, and recall the definitions of $s_{q,j}$, $n_{p,j}$, 
    $n_q^*$, and $t_{q,m}$ from the beginning of this proof. Let 
    $\ell := |q'|$, let $i := e_\beta^{-1}\{\alpha\}$, and let $s$ be an 
    extension of $q'$ such that $s(\ell) = i$. Unwinding the definitions, we 
    obtain the following facts.
    \begin{itemize}
        \item Since $s$ extends $q'$, $s$ forces that $s_j = s_{q,j}$ and 
        $n_j = n_{q,j}$ for all $j \leq k$. As a result, $s$ forces that 
        $t^{n_0}_{s_0}{}^\frown t^{n_1}_{s_1}{}^\frown \ldots{}^\frown 
        t^{n_k}_{s_k} = t_{q,m}$.
        \item Since $\type(C_\alpha[n^*_q], C_\beta[n^*_q]) = t_{q,m}$, 
        the above point implies that, for all $j \leq k$, $s$ forces that 
        $\type(C_\alpha[I_j], C_\beta[I_j]) = t^{n_j}_{s_j}$. In particular, 
        it follows that $\alpha > C_\beta(k)$.
        \item By the fact that $q'$ was chosen of minimal length, it follows 
        that $s$ forces that $\ell_k = \ell$.
        \item Since $s(\ell) = i$, $s$ forces that $e_\beta(r(\ell_k)) = 
        e_\beta(i) = \alpha$.
    \end{itemize}
    Altogether, the four points above imply that $s \Vdash \alpha \in N^{<}(\beta)$. 
    However, we also have $s \Vdash \dot{c}(\alpha) = \dot{c}(\beta) = m$, so 
    $s$ forces that $\dot{c}$ is not a proper coloring of $\dot{G}$.
\end{proof}

\section{Type guessing and the Proper Forcing Axiom} \label{pfa_sect}

In this section, we show that $\mathsf{DTG}$ is not a theorem of $\mathsf{ZFC}$, 
as a strong negation of it follows from the Proper Forcing Axiom, $\mathsf{PFA}$ (in fact, from 
$\mathsf{PFA}(\omega_1)$). Recall that $\mathsf{PFA}(\omega_1)$ is the assertion 
that, for every proper forcing poset $\P$ of cardinality $\omega_1$ and every 
collection $\{D_\alpha \mid \alpha < \omega_1\}$ of dense subsets of $\P$, 
there is a filter $G \subseteq \P$ such that, for all $\alpha < \omega_1$, 
$G \cap D_\alpha \neq \emptyset$. Unlike the full $\mathsf{PFA}$, 
$\mathsf{PFA}(\omega_1)$ has no large cardinal strength and can be forced over 
any model of $\mathsf{ZFC}$ (this follows from, e.g., Lemmas 2.4 and 2.5 of 
Chapter VIII of \cite{shelah1998}; see also \cite{dobrinen2023}).

Before we state the main result of this section, we recall the definition 
of \emph{strongly proper} forcing.

\begin{definition}
    Suppose that $\P$ is a forcing poset. 
    \begin{enumerate}
        \item Given a set $M$, a condition $q \in \P$ is 
        \emph{strongly $(M,\P)$-generic} if, for all $r \leq p$, there is a condition 
        $r|M \in M \cap \P$ such that every extension of $r|M$ in $M 
        \cap \P$ is compatible with $r$ in $\P$.
        \item $\P$ is \emph{strongly proper} if, for all sufficiently large 
        regular cardinals $\theta$, there is $x \in H(\theta)$ such that, 
        for every countable $M \prec (H(\theta), \in, x)$ and every 
        $p \in M \cap \P$, there is a strongly $(M,\P)$-generic condition 
        $q \leq p$.
    \end{enumerate}
\end{definition}

\begin{theorem} \label{pfa_thm}
    Suppose that $\pfa(\omega_1)$ holds and $\vec{t} = \langle t_k \mid k < \omega \rangle$ is a 
    sequence of disjoint types such that $\sup\{\depth(t_k) \mid k < \omega\} = \omega$. 
    Then there does not exist a $\vec{t}$-guessing sequence.
\end{theorem}

\begin{proof}
    Fix a $C$-sequence $\vec{C} = \langle C_\alpha \mid \alpha \in \lim(\omega_1) \rangle$. 
    We will define a forcing notion $\P$ of size $\omega_1$, prove that $\P$ is (strongly) proper, and then apply 
    $\pfa(\omega_1)$ to yield a function $f:\lim(\omega_1) \rightarrow \omega$ witnessing that 
    $\vec{C}$ is not a $\vec{t}$-guessing sequence. 
    
    Let $\langle (d_k, n_k) \mid 
    k < \omega \rangle$ be such that, for all $k < \omega$, we have 
    $d_k = \depth(t_k)$ and $n_k = \width(t_k)$.
    Conditions in $\P$ are all pairs $p = (x_p, f_p)$ such that
    \begin{enumerate}
        \item $x_p \in [\omega_1]^{<\omega}$;
        \item $f_p$ is a finite partial function from $\lim(\omega_1)$ to $\omega$;
        \item for all $\alpha < \beta$, both in $\dom(f_p)$, if 
        $f_p(\alpha) = f_p(\beta) = k$, then 
        \[
          \type(C_\alpha[n_k], C_\beta[n_k]) \notin \{t_k, \bar{t}_k\};
        \]
        \item for all $\delta \in x_p$ and $\beta \in \dom(f_p) \setminus (\delta + 1)$, 
        if $k = f_p(\beta)$, then either
        \begin{enumerate}
            \item $d_k > |C_\beta \cap \delta| + 1$; or 
            \item there is $\alpha \in \dom(f_p) \cap (\delta + 1)$ such that 
            $f_p(\alpha) = k$ and $C_\alpha[n_k] = C_\beta[n_k]$.
        \end{enumerate}
    \end{enumerate}
    If $p,q \in \P$, then $q \leq p$ if and only if $x_q \supseteq x_p$ and 
    $f_q \supseteq f_p$.

    We first establish the following simple claim.

    \begin{claim}
        For all $\alpha \in \lim(\omega_1)$, the set $D_\alpha := \{p \in \P \mid 
        \alpha \in \dom(f_p)\}$ is dense in $\P$.
    \end{claim}

    \begin{proof}
        Fix $\alpha \in \lim(\omega_1)$ and $p \in \P$; we will find $q \leq p$ in $D_\alpha$. 
        To avoid triviality, assume that $p \notin D_\alpha$. Let 
        $\delta := \max(x_p \cap \alpha)$ (or $\delta = 0$ if $x_p \cap \alpha = \emptyset$), 
        and use the assumption about $\vec{t}$ to find $k < \omega$ such that
        \begin{itemize}
            \item $k \notin \mathrm{range}(f_p)$; and
            \item $d_k > |C_\alpha \cap \delta| + 1$.
        \end{itemize}
        Define $q \leq p$ by letting $x_q = x_p$ and $f_q = f_p \cup \{(\alpha, k)\}$. It 
        is routine to verify that $q \in \P$, $q \leq p$, and $q \in D_\alpha$.
    \end{proof}

    We now show that $\P$ is strongly proper. To this end, fix a sufficiently large 
    regular cardinal $\theta$, a countable elementary substructure 
    $M \prec (H(\theta), \in, \P, \vec{C}, \vec{t})$, and a condition $p \in M \cap \P$. Let 
    $\delta_M := M \cap \omega_1$, and define $q \leq p$ by letting $x_q = 
    x_p \cup \{\delta_M\}$ and $f_q = f_p$. 
    
    We claim that $q$ is strongly $(M,\P)$-generic.
    To see this, let $r \leq q$ be arbitrary. We must find a condition 
    $r | M \in M \cap \P$ such that, for any $s \leq r | M$ in $M \cap \P$, 
    $s$ is compatible with $r$. By extending $r$ if necessary, we may assume that 
    $\delta_M \in \dom(f_r)$; let $k^* = f_r(\delta_M)$.
    Let $\gamma := \max((x_r \cup \dom(f_r)) \cap \delta_M)$ 
    (or $\gamma = 0$ if $(x_r \cup \dom(f_r)) \cap \delta_M = \emptyset$). By elementarity 
    of $M$, there exists $\bar{\delta}$ such that $\gamma < \bar{\delta} < \delta_M$ 
    and $C_{\bar{\delta}}[n_{k^*}] = C_{\delta_M}[n_{k^*}]$.
    Define $r|M$ by letting $x_{r|M} = x_r \cap M$ and $f_{r|M} = (f_r \cap M) \cup 
    \{(\bar{\delta}, k^*)\}$. The fact that $r|M \in \P$ follows immediately from 
    the fact that $r \in \P$, $f_{r|M}(\bar{\delta}) = f_r(\delta_M) = k^*$, and 
    $C_{\bar{\delta}}[n_{k^*}] = C_{\delta_M}[n_{k^*}]$.

    Now suppose that $s \leq r|M$, with $s \in M \cap \P$. To show that $s$ and $r$ are 
    compatible, it suffices to show that $(x_s \cup x_r, f_s \cup f_r) \in \P$. 
    Item (1) in the definition of $\P$ is immediate, and item (2) follows from 
    the fact that $f_s \supseteq f_{r|M} \supseteq f_r \cap M$. Let us now verify 
    item (3). Because $r$ and $s$ are each in $\P$, it suffices to consider pairs 
    $\alpha < \beta$ such that $\alpha \in \dom(f_s)$ and $\beta \in \dom(f_r) 
    \setminus \delta_M$. Fix such $\alpha < \beta$, and suppose that 
    $f_s(\alpha) = f_r(\beta) = k$. Suppose first that $\beta = \delta_M$, so $k = k^*$. 
    Then, since $f_s(\bar{\delta}) = k^*$ and $C_{\bar{\delta}}[n_{k^*}] = C_{\delta_M}[n_{k^*}]$, 
    it follows from the fact that $s$ is a condition that
    \[
      \type(C_\alpha[n_{k^*}], C_{\delta_M}[n_{k^*}]) = 
      \type(C_\alpha[n_{k^*}], C_{\bar{\delta}}[n_{k^*}]) \notin \{t_k, \bar{t}_k\}.
    \]

    Next suppose that $\beta > \delta_M$. Since $\delta_M \in x_q \subseteq x_r$ and 
    $r$ satisfies requirement (4) in the definition of $\P$, we are in one of two cases. 
    If there is $\beta' \in \dom(f_r) \cap (\delta_M + 1)$ such that 
    $f_r(\beta') = k$ and $C_{\beta'}[n_k] = C_\beta[n_k]$, then we can reach the desired 
    conclusion exactly as in the case in which $\beta = \delta_M$. So assume now 
    that $d_k > |C_\beta \cap \delta_M| + 1$. In particular, 
    $C_\beta(d_k - 1) \geq \delta_M$, so if it were the case that 
    \[
      \type(C_\alpha[n_k], C_\beta[n_k]) \in \{t_k, \bar{t}_k\},
    \]
    then it would need to be the case that $\alpha > \delta_M$. Since $\alpha < \delta_M$, 
    we again reach our desired conclusion.

    We finally verify item (4). Again since $r$ and $s$ are both in $\P$, it suffices to 
    consider pairs $\delta \in x_s$ and $\beta \in \dom(f_r) \setminus \delta_M$. 
    Fix such a $\delta$ and $\beta$, and let $k = f_r(\beta)$. Suppose first that 
    $\beta = \delta_M$, and hence $k = k^*$. If $\delta < \bar{\delta}$, then applying 
    requirement (4) to $\delta$ and $\bar{\delta}$ in the condition $s$ yields the 
    desired conclusion. If $\bar{\delta} \leq \delta$, then $\bar{\delta}$ is a witness to 
    option (b) in requirement (4). 

    Suppose now that $\beta > \delta_M$. Applying requirement (4) to $\delta_M$ and 
    $\beta$ in $r$ yields one of two options. If option (b) of requirement (4) holds, 
    then we can proceed exactly as in the case in which $\beta = \delta_M$ to reach our 
    desired conclusion. If, on the other hand, $d_k > |C_\beta \cap \delta_M| + 1$, then the 
    fact that $\delta < \delta_M$ immediately yields $d_k > |C_\beta \cap \delta| + 1$, and 
    we are done. This completes the verification that $s$ and $r$ are compatible, and hence 
    the proof that $\P$ is strongly proper.

    To complete the proof of the theorem, apply $\pfa(\omega_1)$ to the poset $\P$ and the dense 
    open sets $\{D_\alpha \mid \alpha \in \lim(\omega_1)\}$ to yield a filter $G \subseteq \P$ such 
    that, for all $\alpha \in \lim(\omega_1)$, $G \cap D_\alpha \neq \emptyset$. 
    Let $f = \bigcup \{f_p \mid p \in G\}$. Then $f : \lim(\omega_1) \rightarrow \omega$ witnesses that 
    $\vec{C}$ is not a $\vec{t}$-guessing sequence.
\end{proof}

\section{Open questions} \label{q_sec}

We conclude with a few questions that remain open. First, note 
that, in Section \ref{type_sec}, we show that disjoint type 
guessing for sequences of bounded width is a theorem of $\zfc$, 
while in Section \ref{pfa_sect} we showed that disjoint type 
guessing can consistently fail for sequences of unbounded depth. 
It remains unclear what the situation is for sequences of unbounded 
width but bounded depth. In particular, we ask the following question.

\begin{question}
  Suppose that $\vec{t} = \langle t_k \mid k < \omega \rangle$ is a 
  sequence of disjoint types such that $\sup\{\depth(t_k) \mid k < 
  \omega\} < \omega$. Must there exist a $\vec{t}$-guessing sequence?
\end{question}

Next, we do not know if $\mathsf{DTG}$ is necessary for the 
conclusion of Theorem \ref{growth_rate_thm}.

\begin{question}
  Must it be the case that, in the forcing extension by a single 
  Cohen real, the following statement is true: For every function 
  $f:\omega \to \omega$, there is an HM graph $G$ such that 
  $f_G(k) > f(k)$ for all $3 \leq k < \omega$?
\end{question}

Our final question concerns a formal strengthening of the notion of 
an HM graph. Let us say that a graph $G$ on $\omega_1$ is a 
\emph{regressive HM graph} if it is an HM graph and, moreover, 
there are no proper colorings $c:\omega_1 \to \omega_1$ of $G$ 
such that $c(\alpha) < \alpha$ for all $\alpha \in \lim(\omega_1)$.
All of the methods of constructing HM graphs that we examined 
can be modified to yield regressive HM graphs; we ask whether 
the existence of regressive HM graphs is actually a stronger 
statement than the existence of HM graphs.

\begin{question}
    Suppose there is an HM graph. Must there exist a regressive 
    HM graph?
\end{question}

\bibliographystyle{plain}
\bibliography{bib}

\end{document}